\documentclass[11pt,oneside]{amsart}
\usepackage{amssymb}
\usepackage[margin=1in]{geometry}
\usepackage{amsmath,amscd}
\usepackage{hyperref}
\usepackage{fontenc}
\usepackage{textcomp}
\usepackage{comment}
\theoremstyle{plain}

\newtheorem{theorem}{Theorem}[section]
\newtheorem{lemma}[theorem]{Lemma}

\newtheorem{coro}[theorem]{Corollary}

\newtheorem{prop}[theorem]{Proposition}

\numberwithin{equation}{section}
\newtheorem{rmk}{Remark}

\newcommand\ff[1]{\mathbb{F}_{#1}}
\newcommand\bound{\mathbf{B}}
\newcommand\tq{\mbox{ } | \mbox{ }}
\newcommand\nn{\mathbb{N}}
\newcommand\fra[2]{\displaystyle\frac{#1}{#2}}
\newcommand\rr{\mathbb{R}}

\theoremstyle{plain}

\begin{document}

\author[1]{Adrien Boyer}\thanks{Weizmann Institute of Science, aadrien.boyer@gmail.com}
\author[2]{Antoine Pinochet Lobos}\thanks{Université d'Aix-Marseille, CNRS UMR7373, a.p.lobos@outlook.com}

\title{An ergodic theorem for the quasi-regular representation of the free group}

\begin{abstract}
In \cite{BAMU}, an ergodic theorem \`a  la Birkhoff-von Neumann for the action of the fundamental group of a compact negatively curved manifold on the boundary of its universal cover is proved. A quick corollary is the irreducibility of the associated unitary representation. These results are generalized \cite{BOYER} to the context of convex cocompact groups of isometries of a CAT(-1) space, using Theorem 4.1.1 of \cite{ROBLI}, with the hypothesis of non arithmeticity of the spectrum. We prove all the analog results in the case of the free group $\ff{r}$ of rank $r$ even if $\ff{r}$ is not the fundamental group of a closed manifold, and may have an arithmetic spectrum.\end{abstract}

\subjclass[2010]{Primary 37; Secondary 43, 47}


\maketitle

\section{Introduction}

In this paper, we consider the action of the free group $\ff{r}$ on its boundary $\bound$, a probability space associated to the Cayley graph of $\ff{r}$ relative to its canonical generating set. This action is known to be \textit{ergodic} (see for example \cite{FIGAT} and \cite{FIGATPI}), but since the measure is not preserved, no theorem on the convergence of means of the corresponding unitary operators had been proved. Note that a close result is proved in \cite[Lemma 4, Item (i)]{FIGATPI}.\\ We formulate such a convergence theorem in Theorem \ref{II}. We prove it following the ideas of \cite{BAMU} and \cite{BOYER} replacing \cite[Theorem 4.1.1]{ROBLI} by Theorem \ref{I}.
\subsection{Geometric setting and notation}
\label{geomsetting}

We will denote $\ff{r} = \langle a_1,...,a_r\rangle$ the free group on $r$ generators, for $r \geq 2$. For an element $\gamma \in \ff{r}$, there is a unique reduced word in $\{a^{\pm 1}_1,...,a^{\pm 1}_r\}$ which represents it. This word is denoted $\gamma_1 \cdots \gamma_k$ for some integer $k$ which is called the \textit{length} of $\gamma$ and is denoted by $\vert \gamma \vert$. The set of all elements of length $k$ is denoted $S_n$ and is called the \textit{sphere of radius} $k$.
If $u \in \ff{r}$ and $k \geq \vert u \vert$, let us denote ${Pr}_u(k) := \{\gamma \in \ff{r} \tq \vert \gamma \vert = k \mbox{, } u \mbox{ is a prefix of } \gamma\}$. 
\\
Let $X$ be the Cayley graph of $\ff{r}$ with respect to the set of generators $\{a^{\pm 1}_1,...,a^{\pm 1}_r\}$, which is a $2r$-regular tree. We endow it with the (natural) distance, denoted by $d$, which gives length $1$ to every edge ; for this distance, the natural action of $\ff{r}$ on $X$ is isometric and freely transitive on the vertices ;
the space $X$ is uniquely geodesic, the geodesics between vertices being finite sequences of successive edges. We denote by $[x,y]$ the unique geodesic joining $x$ to $y$.
\\
We fix, once and for all, a vertex $ x_0$ in $X$. For $x \in X$, the vertex of $X$ which is the closest to $x$ in $[ x_0,x]$, is denoted by $\lfloor x \rfloor$ ; because the action is free, we can identify $\lfloor x \rfloor$ with the element $\gamma$ that brings $ x_0$ on it, and this identification is an isometry.
\subsubsection*{The Cayley tree and its boundary}

As for any other CAT$(-1)$ space, we can construct a boundary of $X$ and endow it with a distance and a measure. For a general construction, see \cite{BOURD}. The construction we provide here is elementary.

Let us denote by $\bound$ the set of all right-infinite reduced words on the alphabet $\{a^{\pm 1}_1,...,a^{\pm 1}_r\}$. This set is called the \textbf{boundary} of $X$.

We will consider the set $\overline{X} := X \cup \bound$.

For $u = u_1\cdots u_l \in \ff{r}\setminus\{e\}$, we define the sets
$$X_u := \left\{x \in X \tq u \mbox{ is a prefix of } \lfloor x \rfloor \right\}$$
$$\bound_u := \left\{\xi \in \bound \tq u \mbox{ is a prefix of } \xi \right\}$$
$$C_u := X_u \cup \bound_u$$

We can now define a natural topology on $\overline{X}$ by choosing as a basis of neighborhoods

\begin{enumerate} \item for $x \in X$, the set of all neighborhoods of $x$ in $X$
\item for $\xi \in \bound$, the set $\left\{ C_u \tq u \mbox{ is a prefix of } \xi\right\}$
\end{enumerate}

For this topology, $\overline{X}$ is a compact space in which the subset $X$ is open and dense. The induced topology on $X$ is the one given by the distance. Every isometry of $X$ continuously extend to a homeomorphism of $\overline{X}$.

\subsubsection*{Distance and measure on the boundary}

For $\xi_1$ and $\xi_2$ in $\bound$, we define the \textbf{Gromov product} of $\xi_1$ and $\xi_2$ with respect to $x_{0}$ by
$$(\xi_1\vert\xi_2) _{x_{0}}:= \sup\left\{k \in \nn \tq \xi_1 \mbox{ and } \xi_2 \mbox{ have a common prefix of length } k\right\}$$ and
$$d_{x_{0}}(\xi_1,\xi_2) := e^{-(\xi_1\vert\xi_2)_{x_{0}}}.$$

Then $d$ defines an ultrametric distance on $\bound$ which induces the same topology ; precisely, if $\xi = u_1u_2u_3 \cdots$, then the ball centered in $\xi$ of radius $e^{-k}$ is just $\bound_{u_1\dots u_k}$.

On $\bound$, there is at most one Borel regular probability measure which is invariant under the isometries of $X$ which fix $ x_0$; indeed, such a measure $\mu_{x_{0}}$ must satisfy $$\mu_{x_{0}}(\bound_{u}) = \frac{1}{2r(2r-1)^{\vert u \vert -1}}$$

and it is straightforward to check that the $\ln(2r-1)$-dimensional Hausdorff measure verifies this property.

If $\xi = u_1\cdots u_n \cdots \in \bound$, and $x,y \in X$, then $\left(d(x,u_1\cdots u_n) - d(y,u_1\cdots u_n)\right)_{n \in \nn}$ is stationary. We denote this limit $\beta_\xi(x,y)$. The function $\beta_\xi$ is called the \textbf{Busemann function} at $\xi$.

Let us denote, for $\xi \in \bound$ and $\gamma \in \ff{r}$ the function $$P(\gamma,\xi) := (2r-1)^{\beta_\xi( x_0,\gamma  x_0)}$$

The measure $\mu_{x_{0}}$ is, in addition, quasi-invariant under the action of $\ff{r}$. Precisely, the Radon-Nikodym derivative  is given for $\gamma\in \Gamma$ and for a.e. $\xi\in \textbf{B}$ by $$\frac{d\gamma_* \mu_{x_{0}}}{d\mu_{x_{0}}} (\xi)= P(\gamma,\xi),$$ where $\gamma_{*}\mu_{x_{0}}(A)=\mu_{x_{0}}(\gamma^{-1}A)$ for any Borel subset $A\subset \textbf{B} $.

\subsubsection*{The quasi-regular representation}

Denote the unitary representation, called the quasi-regular representation of $\ff{r}$ on the boundary of $X$ by  $$\begin{array}{rcl}
\pi : \ff{r} &\rightarrow &\mathcal{U}(L^2(\bound))\\
\gamma &\mapsto &\pi(\gamma)\\
\end{array}$$ defined as
$$\big(\pi(\gamma)g\big)(\xi) := P(\gamma,\xi)^{\frac{1}{2}}g(\gamma^{-1}\xi)$$
 for $\gamma \in \ff{r}$ and for $g \in L^2(\bound)$.
We define the \textit{Harish-Chandra} function
\begin{equation}\label{HCH}
 \Xi(\gamma) :=\langle \pi(\gamma)\textbf{1}_{\textbf{B}},\textbf{1}_{\textbf{B}} \rangle =\int_{\bound} P(\gamma,\xi)^{\frac{1}{2}} d\mu_{x_{0}}(\xi),
 \end{equation}
where $\textbf{1}_{\textbf{B}}$ denotes the characteristic function on the boundary.

For $f \in C(\overline{X})$, we define the operators 
\begin{equation}\label{operators}
M_n(f) : g\in L^2(\bound) \mapsto \fra{1}{\vert S_n \vert} \sum\limits_{\gamma \in S_n} f(\gamma  x_0)\fra{\pi(\gamma) g}{\Xi(\gamma)} \in L^2(\bound).
\end{equation}

We also define the operator 
\begin{equation}
M(f):=m(f_{|_{\bound}})P_{\textbf{1}_{\textbf{B}}}
\end{equation}
where $m(f_{|_{\bound}})$ is the multiplication operator by $f_{|_{\bound}}$ on $L^2(\bound)$, and $P_{\textbf{1}_{\textbf{B}}}$ is the orthogonal projection on the subspace of constant functions.

\subsection*{Results}
The analog of Roblin's equidistribution theorem for the free group is the following.
\begin{theorem} \label{I} We have, in $C(\overline{X}\times\overline{X})^{*}$, the weak-$*$ convergence $$\frac{1}{\vert S_n \vert} \displaystyle\sum_{\gamma \in S_n} D_{\gamma  x_0} \otimes D_{\gamma^{-1}  x_0} \rightharpoonup \mu_{x_{0}} \otimes \mu_{x_{0}}$$ where $D_x$ denotes the Dirac measure on a point $x$.
\end{theorem}

\begin{rmk} It is then straightforward to deduce the weak-$*$ convergence $$\Vert m_{\Gamma} \Vert e^{-\delta n} \displaystyle\sum_{\vert \gamma \vert \leq n} D_{\gamma  x_0} \otimes D_{\gamma^{-1}  x_0} \rightharpoonup \mu_{x_{0}} \otimes \mu_{x_{0}}$$ $m_{\Gamma}$ denoting the Bowen-Margulis-Sullivan measure on the geodesic flow of $SX/\Gamma$ (where $SX$ is the ``unit tangent bundle") and $\delta$ denoting $\ln(2r-1)$, the Hausdorff measure of $\bound$.
\begin{enumerate}
\item Notice that in our case, the spectrum is $\mathbb{Z}$ so the geodesic flow is not topologically mixing, according to \cite{DALBO} or directly by \cite[Ex 1.3]{CHARA}.
\item Notice also that our multiplicative term is different of that of \cite[Theorem 4.1.1]{ROBLI}, which shows that the hypothesis of non-arithmeticity of the spectrum cannot be removed.
\end{enumerate}
\end{rmk}

We use the above theorem to prove the following convergence of operators.
\begin{theorem} \label{II} We have, for all $f$ in $C(\overline{X})$, the weak operator convergence $$ M_n(f) \underset{n\to+\infty}{\longrightarrow} M(f).$$
In other words, we have, for all $f$ in $C(\overline{X})$ and for all $g$, $h$ in $L^2(\bound)$, the convergence
$$\frac{1}{\vert S_{n} \vert}\sum_{\gamma \in S_{n}}f(\gamma x_{0})\frac{\langle \pi(\gamma)g,h\rangle}{\Xi(\gamma)} \underset{n\to+\infty}{\longrightarrow} \langle M(f)g,h \rangle.$$
\end{theorem}

We deduce the irreducibility of $\pi$, and give an alternative proof of this well known result (see \cite[Theorem 5]{FIGAT}).
\begin{coro} The representation $\pi$ is irreducible.
\end{coro}

\begin{proof} Applying Theorem \ref{II} to $f = \textbf{1}_{\overline{X}}$ shows that the orthogonal projection onto the space of constant functions is in the von Neumann algebra associated with $\pi$. Then applying Theorem \ref{II} to $g= \textbf{1}_{\bound}$ shows that the vector $1_{\bound}$ is cyclic. Then, the classical argument of \cite[Lemma 6.1]{GARNC} concludes the proof.
\end{proof}

\begin{rmk} For $\alpha \in \rr^*_+$, let us denote by $W_\alpha$ the wedge of two circles, one of length $1$ and the other of length $\alpha$. Let $p : T_\alpha \twoheadrightarrow W_\alpha$ the universal cover, with $T_\alpha$ endowed with the distance making $p$ a local isometry. Then $\ff{2} \simeq \pi_1(W_\alpha)$ acts freely properly discontinously and cocompactly on the $4$-regular tree $T_\alpha$ (which is a CAT(-1) space) by isometries. For $\alpha \in \rr \setminus \mathbb{Q}$, the analog of Theorem \ref{II} for the quasi-regular representation $\pi_\alpha$ of $\ff{2}$ on $L^2(\partial T_\alpha, \mu_\alpha)$ for a Patterson-Sullivan measure associated to a Bourdon distance is known to hold (\cite{BOYER}) because \cite[Theorem 4.1.1]{ROBLI} is true in this setting. Now if $\alpha_1$ and $\alpha_2$ are such that $\alpha_1 \not = \alpha^{\pm 1}_2$, then the representations $\pi_\alpha$ are not unitarily equivalent (\cite[Theorem 7.5] {GARNC}). For $\alpha \in \mathbb{Q}^*_+ \setminus \{1\}$, it would be interesting to formulate and prove an equidistribution result like Theorem \ref{I} in order to prove Theorem \ref{II} for $\pi_\alpha$.
\end{rmk}

\section{Proofs}
\subsection{Proof of the equidistribution theorem}

For the proof of Theorem \ref{I}, let us denote $$E := \left\{f : C(\overline{X}\times\overline{X}) \tq \fra{1}{\vert S_n \vert} \displaystyle\sum_{\gamma \in S_n} f(\gamma  x_0,\gamma^{-1} x_0) \rightarrow \int_{\overline{X}\times\overline{X}} f d(\mu_{x_{0}} \otimes \mu_{x_{0}})\right\}$$

The subspace $E$ is clearly closed in $C(\overline{X} \times \overline{X})$ ; it remains only to show that it contains a dense subspace of it.

Let us define a modified version of certain characteristic functions : for $u \in \ff{r}$ we define 
$$\chi_u(x) := 
\left\{\begin{array}{ccl}
\max\{1 - d_X(x,C_u),0\} &\mbox{ if } &x \in X\\
0 &\mbox{ if } &x \in \bound \setminus \bound_u\\
1 &\mbox{ if } &x \in \bound_u\\
\end{array}\right.$$

It is easy to check that he function $\chi_u$ is a continuous function which coincides with $\chi_{C_u}$ on $\ff{r} x_0$ and $\bound$.

The proof of the following lemma is straightforward.

\begin{lemma} \label{subalgebra} Let $u \in \ff{r}$ and $k \geq \vert u \vert$, then $\chi_u - \displaystyle\sum\limits_{\gamma \in Pr_u(k)} \chi_\gamma$ has compact support included in $X$.
\end{lemma}

\begin{prop} \label{algebra}The set $\chi :=\{ \chi_u \tq u \in \ff{r} \setminus\{e\}\}$ separates points of $\bound$, and the product of two such functions of $\chi$ is either in $\chi$, the sum of a function in $\chi$ and of a function with compact support contained in $X$, or zero.
\end{prop}

\begin{proof} It is clear that $\chi$ separates points. It follows from Lemma \ref{subalgebra} that $\chi_u \chi_v = \chi_v$ if $u$ is a proper prefix of $v$, that $\chi_u^2 - \chi_u$ has compact support in $X$, and that $\chi_u \chi_v = 0$ if none of $u$ and $v$ is a proper prefix of the other.
\end{proof}

\begin{prop}\label{combinroblin} The subspace $E$ contains all functions of the form $\chi_u \otimes \chi_v$.
\end{prop}

\begin{proof} We make the useful observation that $$\fra{1}{\vert S_n \vert} \displaystyle\sum_{\gamma \in S_n} (\chi_u \otimes \chi_v)(\gamma  x_0,\gamma^{-1} x_0) = \fra{\vert S^{u,v}_n \vert}{\vert S_n \vert}$$ where $S^{u,v}_n$ is the set of reduced words of length $n$ with $u$ as a prefix and $v^{-1}$ as a suffix. We easily see that this set is in bijection with the set of all reduced words of length $n - (\vert u \vert + \vert v \vert)$ that do not begin by the inverse of the last letter of $u$, and that do not end by the inverse of the first letter of $v^{-1}$. So we have to compute, for $s,t \in \{a^{\pm 1}_1,...,a^{\pm 1}_r\}$ and $m \in \nn$, the cardinal of the set $S_m(s,t)$ of reduced words of length $m$ that do not start by $s$ and do not finish by $t$. 

Now we have $$S_m = S_m(s,t) \cup \{ x \tq \vert  x \vert = m \mbox{ and starts by }s\} \cup \{ x \tq \vert  x \vert = m \mbox{ and ends by }t\}.$$

Note that the intersection of the two last sets is the set of words both starting by $s$ and ending by $t$, which is in bijection with $S_{m-2}(s^{-1},t^{-1})$.

We have then the recurrence relation :

\vspace{0.3cm}

$\begin{array}{rcl}
\vert S_m(s,t) \vert &= &2r(2r-1)^{m-1} - 2(2r-1)^{m-1} + \vert S_{m-2}(s^{-1},t^{-1}) \vert\\
&= &2(r-1)(2r-1)^{m-1} + 2(r-1)(2r-1)^{m-3} + \vert S_{m-4}(s,t) \vert\\
&= &(2r-1)^{m}\fra{2(r-1)\left((2r-1)^2 + 1\right)}{(2r-1)^{3}} + \vert S_{m-4}(s,t) \vert\\
\end{array}$.

\vspace{0.3cm}

We set $C := \frac{2(r-1)\left((2r-1)^2 + 1\right)}{(2r-1)^{3}}$, $n = 4k+j$ with $0\leq j \leq 3$ and we obtain 

\vspace{0.3cm}

$\begin{array}{rcl}
\vert S^{s,t}_{4k + j} \vert &= &C(2r-1)^{4k+j} + \vert S^{s,t}_{4(k-1)+j} \vert\\
&= &C(2r-1)^{4k+j} + C(2r-1)^{4(k-1)+j}+ \vert S^{s,t}_{4(k-2) + j} \vert\\
\\
&= &C\displaystyle\sum^{k}_{i=1} (2r-1)^{4i+j} + \vert S^{s,t}_{j} \vert\\
&= &C(2r-1)^{4 + j} \fra{(2r-1)^{4k} - 1}{(2r-1)^4 - 1} +  \vert S_{j}(s,t) \vert\\
\\
&= &(2r-1)^{1 + j}\fra{(2r-1)^{4k} - 1}{2r} + \vert S_{j}(s,t) \vert\\ 
\end{array}$

Now we can compute 

\vspace{0.3cm}

$\begin{array}{rcl}
\fra{\vert S^{u,v}_{4k+j} \vert}{\vert S_{4k+j} \vert} &= &\fra{\left\vert S_{4k+j - (\vert u \vert + \vert v \vert)}(u_{\vert u \vert},v^{-1}_{\vert v \vert}) \right\vert}{\vert S_{4k+j} \vert}\\
\\
&= &\fra{(2r-1)^{1 + j}\fra{(2r-1)^{4k - (\vert u \vert + \vert v \vert)} - 1}{2r} + \left\vert S_{j}(u_{\vert u \vert},v^{-1}_{\vert v \vert}) \right\vert}{2r(2r-1)^{4k+j - 1}}\\
\\
&= &\fra{1}{2r(2r-1)^{\vert u \vert - 1}}\fra{1}{2r(2r-1)^{\vert v \vert - 1}} + o(1)\\
\\
&= &\mu_{x_{0}}(\bound_{u}) \mu_{x_{0}}(\bound_{v}) + o(1)\\
\end{array}$

\vspace{0.3cm}

when $k \to \infty$, and this proves the claim.
\end{proof}

\begin{coro} The subspace $E$ is dense in $C(\overline{X}\times\overline{X})$.
\end{coro}

\begin{proof} Let us consider $E'$, the subspace generated by the constant functions, the functions which can be written as $f\otimes g$ where $f,g$ are continuous functions on $\overline{X}$ and such that one of them has compact support included in $X$, and the functions of the form $\chi_u \otimes \chi_v$. By Proposition \ref{algebra}, it is a subalgebra of $C(\overline{X}\times\overline{X})$ containing the constants and separating points, so by the Stone-Weierstra\ss \mbox{ }theorem, $E'$ is dense in $C(\overline{X}\times\overline{X})$. Now, by Proposition \ref{combinroblin}, we have that $E' \subseteq E$, so $E$ is dense as well.
\end{proof}

\subsection{Proof of the ergodic theorem}
\label{sectionergo}

The proof of Theorem \ref{II} consists in two steps:

\textbf{Step 1}: Prove that the sequence $M_n$ is bounded in $\mathcal{L}(C(\overline{X}),\mathcal{B}(L^2(\bound)))$.

\textbf{Step 2}: Prove that the sequence converges on a dense subset.

\subsubsection{Boundedness}
In the following $\textbf{1}_{\overline{X}}$ denotes the characteristic function of $\overline{X}$. Define $$F_n := \left[M_n(\textbf{1}_{\overline{X}})\right]\textbf{1}_{\bound}.$$ We denote by $\Xi(n)$ the common value of $\Xi$ on elements of length $n$.

\begin{coro} The function $\xi \mapsto \sum\limits_{\gamma \in S_n} \left(P(\gamma,\xi)\right)^{\frac{1}{2}}$ is constant equal to $\vert S_n  \vert \times\Xi(n)$.
\end{coro}

\begin{proof} This function is constant on orbits of the action of the group of automorphisms of $X$ fixing $ x_0$. Since it is transitive on $\bound$, the function is constant. By integrating, we find  

\vspace{0.3cm}

$\begin{array}{rcl}
\displaystyle\sum\limits_{\gamma \in S_n} \left(P(\gamma,\xi)\right)^{\frac{1}{2}} &= &\displaystyle\int_{\bound} \sum\limits_{\gamma \in S_n} \left(P(\gamma,\xi)\right)^{\frac{1}{2}} d\mu_{x_{0}}(\xi)\\
&= &\displaystyle\sum\limits_{\gamma \in S_n} \displaystyle\int_{\bound} \left(P(\gamma,\xi)\right)^{\frac{1}{2}} d\mu_{x_{0}}(\xi)\\
&= &\displaystyle\sum\limits_{\gamma \in S_n} \Xi(n)\\
&= &\vert S_n \vert \Xi(n),\\
\end{array}$
\end{proof}

\begin{lemma} The function $F_n$ is constant, equal to $\textbf{1}_{\bound}$.
\end{lemma}

\begin{proof} Because $\Xi$ depends only on the length, we have that

\vspace{0.3cm}

$\begin{array}{rcl}
F_n(\xi) &:= &\fra{1}{\vert S_n\vert} \sum\limits_{\gamma \in S_n} \fra{\left(P(\gamma,\xi)\right)^{\frac{1}{2}}}{\Xi(\gamma)}\\
&= &\fra{1}{\vert S_n \vert \Xi(n)} \sum\limits_{\gamma \in S_n} \left(P(\gamma,\xi)\right)^{\frac{1}{2}}\\
&=&1,
\end{array}$

\vspace{0.3cm}

and the proof is done.
\end{proof}
It is easy to see that $M_n(f)$ induces continuous linear transformations of $L^1$ and $L^\infty$, which we also denote by $M_n(f)$.

\begin{prop} 
\label{bornitude} The operator $M_n(\textbf{1}_{\overline{X}})$, as an element of $\mathcal{L}(L^{\infty}, L^{\infty})$, has norm $1$; as an element of $\mathcal{B}(L^2(\bound))$, it is self-adjoint.
\end{prop}

\begin{proof} Let $h \in L^{\infty}(\bound)$. Since $M_n(\textbf{1}_{\overline{X}})$ is positive, we have that

\vspace{0.3cm}

$\begin{array}{rcl}
\left\Vert \left[M_n(\textbf{1}_{\overline{X}})\right]h \right\Vert_{\infty} &\leq &\left\Vert \left[M_n(\textbf{1}_{\overline{X}})\right]\textbf{1}_{\bound} \right\Vert_{\infty} \left\Vert h \right\Vert_{\infty}\\
&= &\left\Vert F_n \right\Vert_{\infty} \left\Vert h \right\Vert_{\infty}\\
&= &\Vert h\Vert_\infty\\
\end{array}$

so that $\Vert M_n(\textbf{1}_{\overline{X}})\Vert_{\mathcal{L}(L^{\infty},L^{\infty})} \leq 1$.

The self-adjointness follows from the fact that $\pi(\gamma)^* = \pi(\gamma^{-1})$ and that the set of summation is symmetric.
\end{proof}

Let us briefly recall one useful corollary of Riesz-Thorin's theorem :

Let $(Z,\mu)$ be a probability space.

\begin{prop} \label{rieszthorin} Let $T$ be a continuous operator of $L^1(Z)$ to itself such that the restriction $T_2$ to $L^2(Z)$ (resp. $T_\infty$ to $L^\infty(Z)$) induces a continuous operator of $L^2(Z)$ to itself (resp. $L^\infty(Z)$ to itself).

Suppose also that $T_2$ is self-adjoint, and assume that $\Vert T_\infty \Vert_{\mathcal{L}(L^\infty(Z),L^\infty(Z))} \leq 1$.

Then $\Vert T_2 \Vert_{\mathcal{L}(L^2(Z),L^2(Z))} \leq 1$.
\end{prop}

\begin{proof} Consider the adjoint operator $T^*$ of $(L^1)^* = L^\infty$ to itself. We have that $$\Vert T^* \Vert_{\mathcal{L}(L^\infty,L^\infty)} = \Vert T \Vert_{\mathcal{L}(L^1(Z),L^1(Z))}.$$

Now because $T_2$ is self-adjoint, it is easy to see that $T^* = T_\infty$. This implies $$1 \geq \Vert T^* \Vert_{\mathcal{L}(L^\infty,L^\infty)} = \Vert T \Vert_{\mathcal{L}(L^1(Z),L^1(Z))}.$$

Hence the Riesz-Thorin's theorem gives us the claim.

\end{proof}

\begin{prop} \label{boundedness}The sequence $\left(M_n\right)_{n \in \nn}$ is bounded in $\mathcal{L}(C(\overline{X}),\mathcal{B}(L^2(\bound)))$.
\end{prop}

\begin{proof} Because $M_n(f)$ is positive in $f$, we have, for every positive $g \in L^2(\bound)$, the inequality

\vspace{0.2cm}

$- \Vert f \Vert_\infty [M_n(\textbf{1}_{\overline{X}})]g \leq [M_n(f)]g \leq \Vert f \Vert_\infty [M_n(\textbf{1}_{\overline{X}})]g$

\vspace{0.2cm}

from which we deduce, for every $g \in L^2(\bound)$

\vspace{0.2cm}

$\begin{array}{rcl}
\Vert [M_n(f)]g \Vert_{L^2} &\leq &\Vert f \Vert_{\infty} \Vert [M_n(\textbf{1}_{\overline{X}})]g \Vert_{L^2}\\
&\leq &\Vert f \Vert_{\infty} \mbox{ } \Vert M_n(\textbf{1}_{\overline{X}}) \Vert_{\mathcal{B}(L^2)}  \mbox{ } \Vert g \Vert_{L^2}\\
\end{array}$

\vspace{0.2cm}

which allows us to conclude that 

\vspace{0.2cm}

$\Vert M_n(f)\Vert_{\mathcal{B}(L^2)} \leq \Vert M_n(\textbf{1}_{\overline{X}})\Vert_{\mathcal{B}(L^2)} \Vert f \Vert_{\infty}$.

\vspace{0.2cm}

This proves that $\Vert M_n \Vert_{\mathcal{L}(C(\overline{X}),\mathcal{B}(L^2))} \leq \Vert M_n(\textbf{1}_{\overline{X}}) \Vert_{\mathcal{B}(L^2)}$.

\vspace{0.2cm}

Now, it follows from Proposition \ref{bornitude} and Proposition \ref{rieszthorin} that the sequence $(M_n(\textbf{1}_{\overline{X}}))_{n \in \nn}$ is bounded by $ 1$ in $\mathcal{B}(L^2)$, so we are done.
\end{proof}

\subsubsection{Estimates for the Harish-Chandra function}

The values of the Harish-Chandra are known (see for example \cite[Theorem 2, Item (iii)]{FIGAT}). We provide here the simple computations we need.

We will calculate the value of 

$$\langle \pi(\gamma)\textbf{1}_{\bound}, \textbf{1}_{\bound_u}\rangle =\displaystyle\int_{\bound_u} P(\gamma,\xi)^{\frac{1}{2}} d\mu_{x_{0}}(\xi).$$

\begin{lemma} \label{harish1}Let $\gamma = s_1\cdots s_n \in \ff{r}$. Let $l \in \{1,...,\vert \gamma \vert\}$, and $u = s_1\cdots s_{l-1}t_l t_{l+1}\cdots t_{l+k}$\footnote{For $l=1$, $s_1\cdots s_{l-1}$ is $e$ by convention.}, with $t_l \not = s_l$ and $k \geq 0$, be a reduced word. Then  $$\langle \pi(\gamma)\textbf{1}_{\bound}, \textbf{1}_{\bound_u}\rangle = \fra{1}{2r(2r-1)^{\frac{\vert \gamma \vert}{2}+k}}$$ and $$\langle \pi(\gamma)\textbf{1}_{\bound}, \textbf{1}_{\bound_\gamma}\rangle = \fra{2r-1}{2r(2r-1)^{\frac{\vert \gamma \vert}{2}}}$$
\end{lemma}

\begin{proof} The function $\xi \mapsto \beta_\xi( x_0,\gamma  x_0)$ is constant on $\bound_u$ equal to $ 2(l-1) - \vert \gamma \vert$.

So $\langle \pi(\gamma)\textbf{1}_{\bound}, \textbf{1}_{\bound_u}\rangle$ is the integral of a constant function:
$$\begin{array}{rcl}
\displaystyle\int_{\bound_u} P(\gamma,\xi)^{\frac{1}{2}} d\mu_{x_{0}}(\xi) &= &\mu_{x_{0}}(\bound_u) \mbox{ } e^{\log(2r-1)\left((l-1) - \frac{\vert \gamma \vert}{2}\right)}\\
&= &\fra{1}{2r(2r-1)^{\frac{\vert \gamma \vert}{2}+k}}\cdot\\
\end{array}$$

The value of $\langle \pi(\gamma) \textbf{1}_\bound, \textbf{1}_{\bound_\gamma} \rangle$ is computed in the same way.
\end{proof}

\begin{lemma}\textit{(The Harish-Chandra function)}

Let $\gamma = s_1\cdots s_n$ in $S_{n}$ written as a reduced word. We have that  $$\Xi(\gamma) = \left(1+\fra{r-1}{r}\vert\gamma\vert\right)(2r-1)^{-\frac{\vert\gamma\vert}{2}}.$$
\end{lemma}

\begin{proof} We decompose $\bound$ into the following partition: $$\bound = \displaystyle\bigsqcup\limits_{u_1 \not = s_1} \bound_{u_1} \sqcup \left(\bigsqcup\limits^{\vert \gamma \vert}_{l = 2} \bigsqcup\limits_{\substack{u = s_1\cdots s_{l-1}t_l \\ t_l \not \in \{s_l, (s_{l-1})^{-1}\}}} \bound_u \right) \sqcup \bound_\gamma$$

and Lemma \ref{harish1} provides us the value of the integral on the subsets forming this partition. A simple calculation yields the announced formula.
\end{proof}

The proof of the following lemma is then obvious :

\begin{lemma} \label{harishestimate} If $\gamma, w \in \ff{r}$ are such that $w$ is not a prefix of $\gamma$, then there is a constant $C_w$ not depending on $\gamma$ such that $$\fra{\langle \pi(\gamma)\textbf{1}_{\bound}, \textbf{1}_{\bound_w}\rangle}{\Xi(\gamma)} \leq \fra{C_w}{\vert \gamma \vert}.$$
\end{lemma}

\subsubsection{Analysis of matrix coefficients} The goal of this section is to compute the limit of the \textit{matrix coefficients} $\langle M_n(\chi_u) \textbf{1}_{\bound_v}, \textbf{1}_{\bound_w}\rangle$.

\begin{lemma} \label{matrix1} Let $u, w \in \ff{r}$ such that none of them is a prefix of the other (i.e. $\bound_u \cap \bound_w = \emptyset$). Then $$\displaystyle\lim\limits_{n \to \infty} \langle M_n(\chi_u) \textbf{1}_\bound, \textbf{1}_{\bound_w}\rangle = 0$$
\end{lemma}

\begin{proof} Using Lemma \ref{harishestimate}, we get $$\begin{array}{rcl}
\langle M_n(\chi_u) \textbf{1}_{\bound}, \textbf{1}_{\bound_w}\rangle &= &\fra{1}{\vert S_n \vert} \displaystyle\sum_{\gamma \in S_n} \chi_u(\gamma x_0) \fra{\langle\pi(\gamma)\textbf{1}_{\bound}, \textbf{1}_{\bound_w}\rangle}{\Xi(\gamma)}\\
&= &\fra{1}{\vert S_n \vert} \displaystyle\sum_{\gamma \in C_u \cap S_n} \fra{\langle \pi(\gamma)\textbf{1}_{\bound}, \textbf{1}_{\bound_w}\rangle}{\Xi(\gamma)}\\
&\leq &\fra{1}{\vert S_n \vert} \displaystyle\sum_{\gamma \in C_u \cap S_n} \fra{C_w}{\vert \gamma \vert}\\
&= &O\left(\fra{1}{n}\right)\\
\end{array}$$
\end{proof}

\begin{lemma} \label{matrix2} Let $u, v \in \ff{r}$. Then $$\displaystyle\limsup\limits_{n \to \infty} \langle M_n(\chi_u) \textbf{1}_{\bound_v}, \textbf{1}_\bound\rangle \leq \mu_{x_{0}}(\bound_u)\mu_{x_{0}}(\bound_v)$$
\end{lemma}

\begin{proof}$$\begin{array}{rcl}
\langle M_n(\chi_u) \textbf{1}_{\bound_v}, \textbf{1}_{\bound}\rangle &= &\langle M_n(\chi_u)^* \textbf{1}_{\bound}, \textbf{1}_{\bound_v}\rangle\\
&= &\fra{1}{\vert S_n \vert} \displaystyle\sum_{\gamma \in S_n} \chi_u(\gamma^{-1} x_0) \fra{\langle\pi(\gamma)\textbf{1}_{\bound}, \textbf{1}_{\bound_v}\rangle}{\Xi(\gamma)}\\
&\leq &\fra{1}{\vert S_n \vert} \displaystyle\sum_{\gamma \in S_n} \chi_u(\gamma^{-1} x_0) \chi_v(\gamma x_0)\\
& &+\mbox{ }\fra{1}{\vert S_n \vert} \displaystyle\sum_{\substack{\gamma \in S_n \\ \gamma \not \in C_v}} \chi_u(\gamma^{-1} x_0) \fra{\langle \pi(\gamma)\textbf{1}_{\bound}, \textbf{1}_{\bound_v}\rangle}{\Xi(\gamma)}\\
&= &\fra{1}{\vert S_n \vert} \displaystyle\sum_{\gamma \in S_n} \chi_u(\gamma^{-1} x_0) \chi_v(\gamma x_0)\\
& &+\mbox{ }O\left(\fra{1}{n}\right)\\
\end{array}$$

Hence, by taking the $\limsup$ and using Theorem \textbf{I}, we obtain the desired inequality.
\end{proof}

\begin{prop} \label{matrix3} For all $u, v, w \in \ff{r}$, we have $$\displaystyle\lim\limits_{n \to \infty} \langle M_n(\chi_u) \textbf{1}_{\bound_v}, \textbf{1}_{\bound_w}\rangle = \mu_{x_{0}}(\bound_u \cap \bound_w) \mu_{x_{0}}(\bound_v)$$
\end{prop}

\begin{proof} We first show the inequality $$\displaystyle\limsup\limits_{n \to \infty} \langle M_n(\chi_u) \textbf{1}_{\bound_v}, \textbf{1}_{\bound_w}\rangle \leq \mu_{x_{0}}(\bound_u \cap \bound_w) \mu_{x_{0}}(\bound_v) .$$ If none of $u$ and $w$ is a prefix of the other, we have nothing to do according to Lemma \ref{matrix1}. Let us assume that $u$ is a prefix of $w$ (the other case can be treated analogously). We have, by Lemma \ref{matrix2}, that

$$\begin{array}{rcl}
\mu_{x_{0}}(\bound_w)\mu_{x_{0}}(\bound_v) &\geq &\displaystyle\limsup\limits_{n \to \infty} \langle M_n(\chi_w) \textbf{1}_{\bound_v}, \textbf{1}_\bound\rangle\\
&\geq &\displaystyle\limsup\limits_{n \to \infty} \langle M_n(\chi_w) \textbf{1}_{\bound_v}, \textbf{1}_{\bound_w}\rangle\\
&\geq &\displaystyle\limsup\limits_{n \to \infty} \langle M_n(\chi_w) \textbf{1}_{\bound_v}, \textbf{1}_{\bound_w}\rangle + \sum\limits_{\gamma \in Pr_{u}(\vert w \vert) \setminus \{w\}} \displaystyle\limsup\limits_{n \to \infty} \langle M_n(\chi_\gamma) \textbf{1}_{\bound_v}, \textbf{1}_{\bound_w}\rangle\\
&= &\displaystyle\limsup\limits_{n \to \infty} \langle M_n(\chi_u) \textbf{1}_{\bound_v}, \textbf{1}_{\bound_w}\rangle
\end{array}$$

We now compute the expected limit.
Let us define $$S_{u,v,w} := \{(u',v',w') \in \ff{r} \tq \vert u \vert = \vert u' \vert, \vert v \vert = \vert v' \vert, \vert w \vert = \vert w' \vert\}.$$ Then
$$\begin{array}{rll}
1 &= &\displaystyle\liminf\limits_{n \to \infty} \langle M_n(\textbf{1}_{\overline{X}}) \textbf{1}_{\bound}, \textbf{1}_{\bound}\rangle\\
&\leq &\displaystyle\liminf\limits_{n \to \infty} \langle M_n(\chi_u) \textbf{1}_{\bound_v}, \textbf{1}_{\bound_w}\rangle + \displaystyle\sum\limits_{(u',v',w') \in S_{u,v,w} \setminus\{u, v, w\}} \displaystyle\limsup\limits_{n \to \infty} \langle M_n(\chi_{u'}) \textbf{1}_{\bound_{v'}}, \textbf{1}_{\bound_{w'}}\rangle\\
&\leq &\displaystyle\limsup\limits_{n \to \infty} \langle M_n(\chi_u) \textbf{1}_{\bound_v}, \textbf{1}_{\bound_w}\rangle + \displaystyle\sum\limits_{(u',v',w') \in S_{u,v,w} \setminus\{u, v, w\}} \displaystyle\limsup\limits_{n \to \infty} \langle M_n(\chi_{u'}) \textbf{1}_{\bound_{v'}}, \textbf{1}_{\bound_{w'}}\rangle\\
&\leq &\mu_{x_{0}}(\bound_{u} \cap \bound_{w}) \mu_{x_{0}}(\bound_{v}) + \displaystyle\sum\limits_{(u',v',w') \in S_{u,v,w} \setminus \{u,v,w\}} \mu_{x_{0}}(\bound_{u'} \cap \bound_{w'}) \mu_{x_{0}}(\bound_{v'})\\
&= &1\\
\end{array}$$

This proves that all the inequalities above are in fact equalities, and moreover proves that the inequalities $$\displaystyle\liminf\limits_{n \to \infty} \langle M_n(\chi_u) \textbf{1}_{\bound_v}, \textbf{1}_{\bound_w}\rangle \leq \displaystyle\limsup\limits_{n \to \infty} \langle M_n(\chi_u) \textbf{1}_{\bound_v}, \textbf{1}_{\bound_w}\rangle \leq \mu_{x_{0}}(\bound_{u} \cap \bound_{w}) \mu_{x_{0}}(\bound_{v})$$ are in fact equalities. 
\end{proof}

\begin{proof}[Proof of Theorem \ref{II}] Because of the boundedness of the sequence $(M_n)_{n \in \nn}$ proved in Proposition \ref{boundedness}, it is enough to prove the convergence for all $(f,h_1,h_2)$ in a dense subset of $C(\overline{X})\times L^2 \times L^2$, which is what Proposition \ref{matrix3} asserts.
\end{proof}	

\bibliographystyle{alpha}
\bibliography{bibliobord}

\begin{thebibliography}{FTP83}

\bibitem[BM11]{BAMU}
U.~Bader and R.~Muchnik.
\newblock Boundary unitary representations - irreducibility and rigidity.
\newblock {\em Journal of Modern Dynamics}, 5(1):49--69, 2011.

\bibitem[Bou95]{BOURD}
M.~Bourdon.
\newblock Structure conforme au bord et flot g\'eod\'esique d'un
  {CAT}(-1)-espace.
\newblock {\em Enseign. Math}, 2(2):63--102, 1995.

\bibitem[Boy15]{BOYER}
A.~Boyer.
\newblock Equidistribution, ergodicity and irreducibility in {CAT}(-1) spaces.
\newblock {\em arXiv:1412.8229v2}, 2015.

\bibitem[CT01]{CHARA}
C.~Charitos and G.~Tsapogas.
\newblock Topological mixing in {CAT}(-1)-spaces.
\newblock {\em Trans. of the American Math. Society}, 354(1):235--264, 2001.

\bibitem[Dal99]{DALBO}
F.~Dal'bo.
\newblock Remarques sur le spectre des longueurs d'une surface et comptages.
\newblock {\em Bol. Soc. Bras. Math.}, 30(2):199--221, 1999.

\bibitem[FTP82]{FIGAT}
A.~Fig\`a-Talamanca and M.~A. Picardello.
\newblock Spherical functions and harmonic analysis on free groups.
\newblock {\em J. Functional Anal.}, 47:281--304, 1982.

\bibitem[FTP83]{FIGATPI}
A.~Fig\`a-Talamanca and M.~A. Picardello.
\newblock Harmonic analysis on free groups.
\newblock {\em Lecture Notes in Pure and Applied Mathematics}, 87, 1983.

\bibitem[Gar14]{GARNC}
L. Garncarek.
\newblock Boundary representations of hyperbolic groups.
\newblock {\em arXiv:1404.0903}, 2014.

\bibitem[Rob03]{ROBLI}
T.~Roblin.
\newblock {\em Ergodicit\'e et Equidistribution en courbure n\'egative}.
\newblock M\'emoires de la SMF 95, 2003.

\end{thebibliography}

\end{document}